\documentclass[11pt,reqno,letterpaper]{amsart}
\usepackage{color}
\usepackage[colorlinks=true, allcolors=blue,backref=page]{hyperref}
\usepackage{amsmath, amssymb, amsthm}
\usepackage{mathrsfs}
\usepackage{mathtools}
\usepackage[noabbrev,capitalize,nameinlink]{cleveref}
\crefname{equation}{}{}
\usepackage{fullpage}
\usepackage[noadjust]{cite}
\usepackage{graphics}
\usepackage{pifont}
\usepackage{tikz}
\usepackage{bbm}
\usepackage[T1]{fontenc}

\usetikzlibrary{arrows.meta}

\usepackage{environ}
\usepackage{framed}
\usepackage{url}
\usepackage[linesnumbered,ruled,vlined]{algorithm2e}
\usepackage[noend]{algpseudocode}
\usepackage[labelfont=bf]{caption}
\usepackage{cite}
\usepackage{framed}
\usepackage[framemethod=tikz]{mdframed}
\usepackage{appendix}
\usepackage{graphicx}
\usepackage[textsize=tiny]{todonotes}
\usepackage{tcolorbox}
\usepackage{enumerate}
\allowdisplaybreaks[1]
\usepackage{enumerate}
\usepackage{stmaryrd}
\usepackage[margin=1in]{geometry}

\usepackage[shortlabels]{enumitem}
\crefformat{enumi}{#2#1#3}
\crefrangeformat{enumi}{#3#1#4 to~#5#2#6}
\crefmultiformat{enumi}{#2#1#3}
{ and~#2#1#3}{, #2#1#3}{ and~#2#1#3}

\DeclareSymbolFont{symbolsC}{U}{pxsyc}{m}{n}
\SetSymbolFont{symbolsC}{bold}{U}{pxsyc}{bx}{n}
\DeclareFontSubstitution{U}{pxsyc}{m}{n}
\DeclareMathSymbol{\medcircle}{\mathbin}{symbolsC}{7}

\crefname{algocf}{Algorithm}{Algorithms}

\crefname{equation}{}{} 
\AtBeginEnvironment{appendices}{\crefalias{section}{appendix}} 

\usepackage[color,final]{showkeys} 

\colorlet{refkey}{orange!20}
\colorlet{labelkey}{blue!30}

\crefname{algocf}{Algorithm}{Algorithms}

\numberwithin{equation}{section}
\newtheorem{theorem}{Theorem}[section]
\newtheorem{proposition}[theorem]{Proposition}
\newtheorem{lemma}[theorem]{Lemma}

\crefname{claim}{Claim}{Claims}

\newtheorem*{question*}{Question}
\newtheorem{fact}[theorem]{Fact}

\theoremstyle{definition}

\newtheorem*{definition*}{Definition}

\theoremstyle{remark}
\newtheorem*{remark}{Remark}

\newcommand{\mb}{\mathbb}

\newcommand{\mc}{\mathcal}

\newcommand{\wt}{\widetilde}

\newcommand{\imod}[1]{~\mathrm{mod}~#1}

\newcommand{\eps}{\varepsilon}

\let\originalleft\left
\let\originalright\right
\renewcommand{\left}{\mathopen{}\mathclose\bgroup\originalleft}
\renewcommand{\right}{\aftergroup\egroup\originalright}
\newcommand{\tO}{\wt{\Omega}}

\renewcommand{\Re}{\mathrm{Re}}

\allowdisplaybreaks

\newcommand{\ignore}[1]{}

\title{On further questions regarding unit fractions}

\author[A1]{Yang P. Liu}
\address{School of Mathematics, Institute for Advanced Study, Princeton, NJ 08540, USA}
\email{yangpliu@ias.edu}

\author[A2]{Mehtaab Sawhney}
\address{Department of Mathematics, Massachusetts Institute of Technology, Cambridge, MA 02139, USA}
\email{msawhney@mit.edu}

\begin{document}

\begin{abstract}
We prove that a subset $A\subseteq [1, N]$ with 
\[\sum_{n\in A}\frac{1}{n} \ge (\log N)^{4/5 + o(1)}\]
contains a subset $B$ such that 
\[\sum_{n\in B} \frac{1}{n} = 1.\]
Our techniques refine those of Croot and of Bloom. Using our refinements, we additionally consider a number of questions regarding unit fractions due to Erd\H{o}s and Graham.
\end{abstract}

\maketitle

\section{Introduction}\label{sec:intro}
Erd\H{o}s and Graham \cite{EG79} raised a host of questions regarding unit fractions. Among these questions, they asked whether a coloring of $\mb{N} \setminus \{1\}$ into $k$ colors necessarily contains a color class $\mc{S}$ and $B\subseteq \mc{S}$ such that 
\[\sum_{n\in B}\frac{1}{n} = 1.\]

This was answered positively in work of Croot \cite{Cro03}. Croot proved that there exists an absolute constant $b$ such that if the interval $[2,b^k]$ is $k$-colored then there exists a color class which can be used to construct $1$ as a sum of reciprocals. Proving the density analog of Croot's result was also raised in work of Erd\H{o}s and Graham \cite{EG79}. This was recently resolved in work of Bloom \cite{Blo21} who proved that there is an absolute constant $C$ such that if set $A\subseteq [1, N]$ satisfies
\[\sum_{n\in A}\frac{1}{n} \ge \frac{C(\log N)(\log\log\log N)}{\log\log N}\]
then there exists $B\subseteq A$ such that 
\[\sum_{n\in B}\frac{1}{n} = 1.\]
Phrased another way, if $\lambda(N)$ is the maximal $\sum_{n\in A}\frac{1}{n}$ such that $A\subseteq [1,N]$ and $A$ does not contain a subset $B$ such that $\sum_{n\in B}1/n = 1$, then Bloom proves in \cite{Blo21} that $\lambda(N) \ll \frac{(\log N)(\log \log \log N)}{\log \log N}$. This is far from the best known lower bound of $\lambda(N) \gg (\log \log N)^2$; this lower bound is achieved by considering all numbers in $[1, N]$ with a prime factor larger than $N/(\log N)$ (see \cite[Appendix~A]{Blo21}).
Bloom writes regarding \cite{Blo21} that ``we do not believe the methods in this paper alone are strong enough to prove $\lambda(N)\ll (\log N)^{1-c}$ for some $c>0$''. Our main result is that a variation of the Fourier analytic techniques used by Croot \cite{Cro03} and Bloom \cite{Blo21} are sufficiently strong to achieve such a bound.

\begin{theorem}\label{thm:main}
Let $\eps > 0$, with $N$ sufficiently large depending on $\eps$. For any set $A \subseteq [1, N]$ satisfying $\sum_{n \in A} \frac{1}{n} \ge (\log N)^{4/5+\eps}$, there is a subset $B \subseteq A$ with $\sum_{n \in B} \frac{1}{n} = 1$.
\end{theorem}

We furthermore combine the Fourier analytic techniques used to prove \cref{thm:main} with several additional elementary arguments to address various other problems of Erd\H{o}s and Graham \cite{EG79} regarding unit fractions.

The first result considers the number of subsets of $[1,N]$ whose reciprocal sum is $1$. This question appears as Problem \# 297 on the list of Erd\H{o}s as curated by Bloom \cite{BloWeb}. In recent work, Steinerberger \cite{Ste24} proved that $|\mc{A}|\le 2^{.93N}$; in particular this is exponentially smaller than the trivial $2^{N}$ bound. Our additional result provides the log-asymptotic on the number of subsets of $[1, N]$ whose reciprocal sum is $1$.
\begin{theorem}
\label{thm:297}
Let $\lambda^\ast$ be the unique positive real number satisfying
\[ \int_0^1 \frac{e^{-\lambda^{\ast}/x}}{1+e^{-\lambda^{\ast}/x}}\cdot \frac{dx}{x} = 1. \]
Let $\mathcal{A} := \{ S \subseteq [1, N] : \sum_{s \in S} 1/s = 1 \}.$ Then for $\gamma^\ast := \lambda^\ast + \int_0^1 \log(1+e^{-\lambda^\ast/x}) \, dx$, we have that
\[ e^{\gamma^\ast N - o(N)} \le |\mathcal{A}| \le e^{\gamma^\ast N + o(N)}. \]
\end{theorem}
\begin{remark}
Via numerical computation, we find that $\lambda^{\ast}\approx .127191$, $\gamma^\ast\approx .631573$, and $e^{\gamma^\ast} \approx 1.88057$.
\end{remark}

Erd\H{o}s and Graham \cite{EG79} also considered the largest subset of $[1, N]$ which does not an contain a subset with reciprocal sum $1$; this appears as Problem \# 300 in \cite{BloWeb}. From work of Croot \cite{Cro03}, it follows that if $|A|\ge (1-\delta)N$ for $\delta$ sufficiently small there exists such a subset. Our next result is an essentially tight bound on the largest number of elements that a subset $A \subseteq [1, N]$ can contain without having a subset whose reciprocal sum is $1$.
\begin{theorem}\label{thm:300}
Fix $\eps \in (0,1/2)$. For $N$ sufficiently large in term of $\eps$ and $A \subseteq [1, N]$ with $|A|\ge (1-1/e+\eps)N$, there exists a subset $A'$ with 
\[\sum_{n\in A'}\frac{1}{n} = 1.\]
\end{theorem}

Note here the constant $(1-1/e)$ is sharp as for $\gamma > 0$ and $N$ sufficiently large, we have that 
\[\sum_{n =  \lfloor (1/e + \gamma) N\lfloor}^{N}\frac{1}{n}\le \int_{1/e + \gamma/2}^{1}\frac{dx}{x} < 1.\] 
This constant is closely related to work of Croot \cite{Cro01} about how ``narrow'' the range of denominators in Egyptian fraction used to construct $1$ may be. Answering a question of Erd\H{o}s and Graham \cite{EG79}, Croot \cite{Cro01} proved that given any rational number $r$, one may write 
\[r = \frac{1}{x_1} + \cdots + \frac{1}{x_k}\]
with 
\[N<x_1<\cdots <x_k<\bigg(e^{r}N + O_{r}\bigg(\frac{N\log\log N}{\log N}\bigg)\bigg).\]
The proof of \cref{thm:300} in particular embeds a proof of the above when $r=1$.

In this spirit, Erd\H{o}s and Graham \cite{EG79} asked whether given a dense subset of $[1,N]$ one may represent a fraction with small denominator. This appears as Problem \# 310 in \cite{BloWeb}.
\begin{proposition}\label{prop:310}
There exists a constant $C\ge 1$ such that the following holds. Consider $\eps > 0$, $N$ sufficiently large in terms of $\eps$, $A\subseteq [1,N]$ with $|A|\ge \alpha N$, and $\alpha\in [(\log N)^{-1/7+\eps},1/2]$. Then there exist $1 \le s \le t \le \exp(C/\alpha)$ and $B\subseteq A$ such that 
\[\sum_{n\in B}\frac{1}{n} = \frac{s}{t}. \]
\end{proposition}

We give a number of remarks regarding this proposition. We have labeled \cref{prop:310} as a proposition because a rather direct application of \cite[Proposition~1]{Blo21} along with standard estimates (which are present in \cite{Blo21}) immediately demonstrates that one can take $t = O_{\alpha}(1)$, resolving the original conjecture of Erd\H{o}s and Graham. Achieving the precise dependence in \cref{prop:310} (that one can take $\alpha \ge (\log N)^{-1/7+\eps}$, and $t \ll \exp(C/\alpha)$), requires the improved understanding given by the techniques of this paper. In particular, that one can take $\alpha \ge (\log N)^{-1/7+\eps}$ instead of requiring say $\alpha \ge (\log \log N)^{-1}$ essentially corresponds to the logarithmic savings in \cref{thm:main}.

The next remark is that the dependence $t\ll \exp(C/\alpha)$ is essentially sharp. To see this, take $A$ to be the subset of $[N/2,N]$ with all prime factors larger than $\exp(1/\alpha)$. This set can be checked to have density $\gg \alpha$ by standard sieve theory (e.g. \cref{lem:fund-lem}) and as $A$ has reciprocal sum $<1$ any nontrivial sum of reciprocals from a subset of $A$ has denominator at least $\exp(1/\alpha)$. 

We finally end by considering representing fractions by ``low complexity'' unit fractions. Erd\H{o}s and Graham \cite{EG79} considered given a fraction $a/b$ with $1\le a\le b$ both the ``fewest'' number of unit fractions required to represent $a/b$ and the ``smallest size'' denominators required to represent such fractions. In the direction of ``fewest'' number of unit fractions required, Erd\"{o}s \cite{Er50} proved that $\gg \log\log b$ are fractions are required for a typical $a$ while Vose \cite{Vos85} proved that each fraction can be constructed with $\ll \sqrt{\log b}$ many unit fractions. 

We will primarily be concerned with constructing $a/b$ with denominators of ``small size''. Let 
\[D(a,b) = \max_{1\le a\le b-1}\min\bigg\{n_k: \frac{a}{b} = \sum_{j=1}^{k}\frac{1}{n_j}~\text{ and }n_1<\cdots<n_k\bigg\}.\]
Bleicher and Erd\H{o}s \cite{BE76} proved that $D(a,b) \ll b(\log b)^2$ and asked whether $D(a, b) \ll b(\log b)^{1+o(1)}$ in fact holds; this question was reiterated by Erd\H{o}s and Graham \cite{EG79} (Problem \# 305 in \cite{BloWeb}). Yokota resolved this question in a series of papers \cite{Yok86,Yok88,Yok88b}. First, Yokota \cite{Yok86} reduced to the case where $b$ is prime and then via increasingly ingenious elementary arguments proved $D(a,b)\le b(\log b)^{3/2}$ in \cite{Yok88}, and finally in \cite{Yok88b} proved 
\[D(a,b) \ll b(\log b)(\log\log b)^{4}(\log\log\log b)^2.\]

Our Fourier analytic techniques coupled with basic arguments in fact improve this bound slightly.
\begin{theorem}\label{thm:305}
For integers $1 \le a < b$, there are integers \[ 1 < n_1 < n_2 < \dots < n_k \le b(\log b)(\log \log b)^3(\log\log\log b)^{O(1)} \]with $\frac{a}{b} = \frac{1}{n_1} + \frac{1}{n_2} + \dots + \frac{1}{n_k}.$
\end{theorem}

We remark that Bleicher and Erd\H{o}s \cite{BE76b} proved that 
\[D(a,b)\gg b(\log b)(\log\log b)^{1-o(1)}.\]
Finally, Erd\H{o}s and Graham \cite{EG79} (Problem \# 305 in \cite{BloWeb}) raised the related question of studying the values of $t$ for which one can write $1$ as a sum of unit fractions with reciprocal containing $t$ and integers between $t+1$ and $N$. To our knowledge there has been no previous work on this question. We resolve it up to a $(\log\log N)^{O(1)}$ factor in an essentially identical manner to \cref{thm:305}.
\begin{theorem}\label{thm:294}
Let $N\ge 1$ and let $t = t(N)$ be the least integer such that there is no solution to 
\[1 = \frac{1}{n_1} + \cdots + \frac{1}{n_k}\]
with $t=n_1<n_2<\ldots<n_k\le N$. Then 
\[\frac{N}{(\log N)(\log\log N)^3(\log\log\log N)^{O(1)}}\ll t(N)\ll \frac{N}{\log N}.\]
\end{theorem}
\begin{remark}
The upper bound is claimed by Erd\H{o}s and Graham \cite[pg.~35]{EG79} without proof (although very closely related arguments appear in \cite{BE76,BE76b}). We include the simple proof in \cref{sec:294-297} for the reader's convenience.
\end{remark}

\subsection{Previous work}
In this section, we focus on the work of Croot \cite{Cro03} and Bloom \cite{Blo21}. 

While we phrased Croot's result as a coloring result, in reality Croot proves a density result with respect to smooth numbers. In particular, Croot proves that given a subset of $N^{1/10}$-smooth numbers from $[N^{9/10},N]$ with reciprocal sum $\ge 6$ there exists a subset with reciprocal sum equal to $1$. This is sufficient for proving the desired coloring result, as a positive density of integers from  $[N^{9/10},N]$ are $N^{1/10}$--smooth. A crucial realization of Croot \cite{Cro99,Cro01,Cro03} in a series of works regarding unit fractions is that Fourier analytic methods may be applied effectively. The smoothness parameter is used to solve a certain ``inverse'' problem regarding modular inverses. The key issue in passing from coloring to density results is improving the smoothness threshold to $N^{1-o(1)}$.

This key technical improvement of the smoothness upper bound was accomplished in work of Bloom \cite{Blo21}. This allows Bloom to prove a stronger conjecture of Erd\H{o}s and Graham \cite{EG79} that any subset of $\mb{N}$ with positive upper density contains a subset with reciprocal sum equal to $1$.

 The key improvement in our work stems from the precise technical details for how the ``inverse problem'' alluded to above is solved when handling numbers of smoothness $N^{1-o(1)}$. 

\subsection{Proof outline}
We first discuss the proof of the Fourier analytic argument of Croot and Bloom and the optimizations in our approach which allow us to prove \cref{thm:main}. We then outline various additional arguments which are used to prove \cref{thm:294,thm:297,thm:300,thm:305,prop:310}.

\subsubsection{Proof of main proposition}\label{sub:prop-main}
We remark that we use slightly different Fourier analytic arguments for \cref{thm:294,thm:297,thm:305} versus \cref{thm:main,thm:300,prop:310}. This is ultimately due to the fact that in \cref{thm:294,thm:297,thm:305} we may choose the set $A$ from which we seek to construct the unit fraction representations while in \cref{thm:main,thm:300,prop:310} we are given a set $A$. For this outline we will focus on the latter argument; the former is slightly simpler and is used to deliver superior quantitative bounds in \cref{thm:294,thm:297,thm:305}.

Consider a set of integers $A\subseteq [M,N]$ where $M = N^{1-o(1)}$ and $\sum_{n\in A}\frac{1}{n} = \eta \log(N/M)$; e.g. $A$ has logarithmic density $\eta$ in the interval $[M,N]$. Let $\mc{Q}_A$ denote the set of prime power factors dividing at least one element of $A$ and $Q$ denote the least common multiple of the integers in $\mc{Q}_A$. Note that the sum of unit fractions with reciprocals in $A$ may be written as a fraction with denominator $Q$. For the sake of simplicity, consider sampling each element of $A$ with probability $1/2$ to get a random set $B$ and seeking to understand the probability that $\sum_{n\in B}\frac{1}{n}\in \mb{Z}$. Via a Fourier analytic computation we find that 
\[\mb{P}\bigg[\sum_{n\in B}\frac{1}{n}\in \mb{Z}\bigg] = \frac{1}{Q} \sum_{-Q/2<h\le Q/2} \prod_{n \in A} \bigg(\frac{1}{2}+ \frac{e(h/n)}{2}\bigg).\]

Note that the term when $h = 0$ contributes $1/Q$ and via standard Fourier analytic computations one can prove that the contribution of the ``small terms'' (i.e., $|h|\le M/2$) is at least $\ge \frac{3}{4Q}$. This is referred to as the ``major arc'' is prior works. The crucial idea of Bloom in handling the remaining minor arcs is translating whether the expression is small into a divisibility problem regarding $h$. Let $h_n$ denote the unique integer in $(-n/2,n/2]$ which is congruent to $h \imod n$. Then 
\[\prod_{n \in A} \bigg|\bigg(\frac{1}{2}+ \frac{e(h/n)}{2}\bigg)\bigg|\]
is controlled by the number of $h_n$ which are ``relatively small'' (e.g. $|h_n| < K$ for a threshold $K$ with $K \le M < N/2$ which is optimized at the end of the argument). This is equivalent to the interval $(h-K/2,h+K/2]$ containing a multiple of $n$. Instead of handling distinct $n$ separately, we follow Bloom by considering all prime powers $q\in \mc{Q}_A$ and study ``poor'' $q$ such that almost all $n$ in $A$ which are multiples of $q$ have small $h_n$. An easy way that there may be many poor $q$ is that there is an integer $h'$ near $h$ (i.e., $h' \in (h-K/2, h+K/2]$) such that $h'$ is divisible by all the ``poor'' primes powers. The crucial technical statement is proving that this is the only possibility (this is the ``inverse problem'' mentioned earlier).

The proof of the above proceeds in two stages. Fix a poor prime $q$; recall that nearly all multiples of $q$ in $A$ have small $h_n$. First, adapting a lemma of Bloom, we prove that there exists $d_q$ such $qd_q\ge K$ and yet nearly all multiples of $qd_q$ in $A$ have $h_n$ being small; this is the content of \cref{lem:lift}. Note that as $qd_q\ge K$ there is a unique multiple of $qd_q$ in the interval $(h-K/2,h+K/2]$; call this $x_q$. The final point is to establish that $x_q$ agree for different $q$; our improvement over the work of Bloom \cite{Blo21} stems from a simpler method of establishing this fact. 

Note that almost all multiples $n$ of $qd_q$ in $A$ have $h_n$ small. Using elementary sieve theory, this allows us to prove that there are many primes $p$ such that there is $n\in A$ such that $pqd_q$ divides $n$ and $h_n$ is small. Note that this forces $p|x_q$. In fact one can prove that given $q$ and $q'$ that for many primes $p$, we have $p|x_q$ and $p|x_q'$ and therefore $p|(x_q-x_q')$. However note that $|x_q-x_q'|\le K\le N$ and taking sufficiently many primes this establishes $x_q = x_q'$ by the Chinese remainder theorem.

We end with a brief discussion of the precise statement of the main proposition \cref{prop:main}. Bloom's work at various stages seeks to find subsets of low denominators and only then combine them to find a unit fraction. This is due to a number of technical conditions which arise in his work. This need to combine small denominator fractions to create a unit fraction hurts various quantitative aspects and if used na\"{i}vely limits one to bounds of the form $(\log N)/(\log\log N)$. A crucial technical point of \cref{prop:main} is that our statement separates out ``Archimedean'' and ``$p$-adic'' considerations. In particular, even for relatively sparse sets we are able to establish that one can reach ``all possible'' fractions; whether we reach the fraction $1$ is entirely determined by whether the reciprocal sum is sufficiently large. This contrasts with \cite[Proposition~1]{Blo21} which relies on various small prime factors in order to determine which unit fractions can be created.

\subsubsection{Proof of various consequences}\label{sub:applic}
We now briefly discuss the proof for the various applications. For \cref{thm:297}, as mentioned earlier recent work of Steinerberger \cite{Ste24} proved that $|\mc{A}|\le 2^{.93N}$. We note that there is an intuitive way to obtain such an exponential savings. Any set $\mc{S} \subseteq [1, N]$ with reciprocal sum $1$ must satisfy that $\sum_{t\in [N/8,N]\cap \mc{S}}1/t\le 1$. Note that $\sum_{t\in [N/8,N]}1/t>2$ for $N$ larger than an absolute constant. Thus applying the Azuma--Hoeffding inequality (\cref{lem:azuma}) to a random subset $\mc{S}\subseteq [1, N]$ and studying the quantity $\sum_{t\in [N/8,N]\cap \mc{S}}1/t$ immediately provides the desired result. We only consider the elements larger than $N/8$ so that the variance proxy is $\ll \sum_{t\in [N/8,N]}\frac{1}{t^2}\ll 1/N$.

This approach based on general purpose concentration inequalities sheds little light on the numerical constant $\lambda$. For this we rely on the principle of maximum entropy; we wish to assign probabilities $p_i\in [0,1]$ for $i\in [1, N]$ such $\sum_{i\in [1, N]} p_i/i \le 1$ which maximizes $\sum_{i\in [1, N]} H(p_i)$. Via Lagrange multipliers we see that $p_i/(1-p_i) \propto e^{-\lambda n/i}$ or $p_i = e^{-\lambda n/i}/(1+e^{-\lambda n/i})$ yields the maximum entropy. The numerical constant $\lambda$ arises from guaranteeing that $\sum_{i\in [1, N]} p_i/i = 1$. We remark that our upper bound proof is phrased slightly differently but this is essentially cosmetic. With these sampling probabilities we have encoded the ``Archimedean'' portion of the problem and the crucial point for the lower bound is that ``$p$-adic'' constraints can be made negligible. Instead of sampling from all integers $1$ to $N$, we restrict to integers with no prime divisor larger that $N^{1-o(1)}$. It is crucial to have smoothness $N^{1-o(1)}$, as achieved in work of Bloom \cite{Blo21}, as we cannot afford to throw out a positive fraction of numbers and obtain the desired bound. 

We next discuss \cref{thm:294,thm:305}. A rather direct application of the main proposition used to prove \cref{thm:main} gives a bound of $N^{1-o(1)}$ for \cref{thm:294}; in order to get a sharper bound one needs to consider how the threshold $N/\log N$ arises. The upper bound in \cref{thm:294} arises from taking $t$ to be a prime which is a large constant times $N/\log N$. The reason there are no $t < n_1 < \dots < n_k \le N$ with $1 - \frac{1}{t} = \frac{1}{n_1} + \dots + \frac{1}{n_k}$ (which is detailed in \cref{sec:294-297}) is that there are ``not enough'' unit fractions with $t$ dividing their denominator in order to ``eliminate the prime $t$''. Based on this example, we eliminate $t$ in a tailored manner if $t\ge N^{1/5}$. In particular, we prove that given $t\le N$, there exist $t_1,\ldots, t_{\ell}\le (\log N)^{1+o(1)}$, and $a, b$ with $\gcd(a, b) = 1$, such that
\[\frac{a}{b} = 1 + \sum_{j=1}^{\ell}\frac{1}{t_j}, \]
and critically that $t | a$. Given, we have that 
\[1 - \frac{1}{t}\bigg(1 + \sum_{j=1}^{\ell}\frac{1}{t_j}\bigg) = 1 - \frac{a/t}{b} \]
is a number whose denominator has smoothness $(\log N)^{O(1)}$. The Fourier analytic techniques underlying \cref{thm:main} can now be applied once again to create the remaining fraction and we get a $N/(\log N)^{1+o(1)}$-type bound. An essentially identical argument holds in the context of \cref{thm:305}.

Finally we mention \cref{thm:300,prop:310}. These both follow in a straightforward manner from \cref{prop:main} after restricting to the appropriate range and removing numbers which are not appropriately smooth or have too many prime factors.  

We end by mentioning two open problems which seem particularly enticing. The first is proving that one can take $n_k \ll b(\log b)(\log\log b)^{1+o(1)}$ in \cref{thm:305}. By the trick of trying to clear denominators mentioned above, this may amount to analyzing some specially constructed set with smoothness up to $N/(\log N)^{1+o(1)}$. The second is improving \cref{thm:main} to $\lambda(N) \ll (\log N)^{o(1)}$. 

We first explain why our method is naturally limited at smoothness at most $N/(\log N)$. Consider a prime $p$ of size at least $10N/(\log N)$ and notice that independent of $h$ we have 
\[\prod_{\substack{p|n\\ n\in [1,N]}}\bigg|\frac{3}{4} + \frac{e(h/n)}{4}\bigg|\ge N^{-1/5}.\]
This means terms ``divisible by $p$'' cannot account for the corresponding prime $p$ in our analysis. However, our precise argument structure prevents us from achieving this threshold. We believe that one may consider numbers of smoothness up to $N/(\log N)^{2+o(1)}$ (which gives $n_k \ll b(\log b)(\log\log b)^{2+o(1)}$) by working with numbers of the form $pp'd$ where $p \approx N(\log N)^{-2-o(1)}$, $p' \approx (\log N)^{1+o(1)}$, and $d\lesssim (\log n)^{1+o(1)}$ while being $(\log N)^{o(1)}$-smooth. The advantage of considering smooth multiples is that if $n = pp'd$ and $h \mod n$ is ``small'' then one can also deduce that $h \mod n'$ is ``small'' for $n'|n$ and $(n/n')$ sufficiently small by passing to a divisor of $d'$ of $d$. As $d$ is $(\log n)^{o(1)}$ ``smooth'' one can find divisors $(n/n')$ at all scales and this allows one to ``amplify'' numbers with ``small residue'' at large scale at a smaller scale. We omit any further discussion of such an adjustment as a result along these lines substantially lengthens the paper and we believe will be obsolete by an alternate analysis. 

For the second problem, the difficulty is that our only tool for locating various prime factors gets weaker as the density of our set decreases (we ultimately use the Fundamental Lemma of Sieve Theory, \cref{lem:fund-lem}). An analysis achieving $\lambda(N) \ll (\log N)^{o(1)}$ via a Fourier analytic approach seems to require understanding of these exceptional sets beyond simple density consideration. We believe this is a problem of considerable interest. 

\subsection{Notation}
As is standard, we write $e(x) = e(2\pi i x)$. We let $\Re(x)$ denote the real part of $x$. Given $n = p_1^{a_1}\cdots p_k^{a_k}$ where $p_i$ are distinct primes, we let $\omega(n) = k$, $\Omega(n) = \sum_{i=1}^{k} a_i$ and (nonstandardly) define $\tO(n) = \max_i a_i$. We say that a positive integer $n$ is $S$-smooth if every prime power $q | n$ satisfies $q \le S$. Throughout the paper, we will without comment use $p$ to indicate a prime variable and $q$ to denote a prime power variable. 

Furthermore we let $R(A) = \sum_{n \in A} \frac{1}{a}$. Given a set $A$, we let $A_d = \{n \in A : d|n\}$. Throughout let $[a, b] = \{a,\ldots, b\}$ for integers $a \le b$. For a set $\mc{Q}$ (normally of prime powers), we let $[\mc{Q}]$ denote the least common multiple.

We use standard asymptotic notation. Given functions $f=f(n)$ and $g=g(n)$, we write $f=O(g)$, $f\ll g$, or $g\gg f$ to mean that there is a constant $C$ such that $|f(n)|\le Cg(n)$ for sufficiently large $n$. We write $f\asymp g$ or $f=\Theta(g)$ to mean that $f\ll g$ and $g\ll f$, and write $f=o(g)$ to mean $f(n)/g(n)\to0$ as $n\to\infty$.

\subsection{Acknowledgements}
The second author thanks Ashwin Sah for initial discussions regarding these problems. We thank Zachary Chase, Zach Hunter, and Ashwin Sah for comments; in particular we thank Zach Hunter for pointing us to work of Yokota. The first author was supported by NSF DMS-1926686. The second author was supported by NSF Graduate Research Fellowship Program DGE-1745302.

\section{Various preliminaries}\label{sec:prelim}

\subsection{Number theoretic preliminaries}\label{subsec:nt}
We require a series of basic number theoretic preliminaries. These are all completely standard.

The following are immediate consequences of (say) the prime number theorem.
\begin{theorem}\label{thm:Merten}
There exists a constant $M\in (0,1)$ such that the follow holds. We have that 
\[\sum_{p\le N} \frac{1}{p} = \log\log N + M + O((\log N)^{-2}).\]
Furthermore
\[2^{N}\ll \prod_{p\le N}p \le \prod_{q\le N} q\ll 3^{N}.\]
\end{theorem}

We next require that nearly all integers up to $N$ have $\ll \log\log N$ primes factors. 
\begin{lemma}\label{lem:divisor}
We have that 
\[\big|\{n\in [1, N]:\Omega(n)> 5 \log\log N\}\big|\ll N (\log N)^{-3}.\]
\end{lemma}
\begin{proof}
We have that $|\{d|n:n\in [1, N]\}|\le N/d$. Let $k = \lceil 5 \log\log N \rceil$ and let $q$ range over prime powers. For $N$ larger than an absolute constant, we have
\begin{align*}
\big|\{n\in [1, N]:\Omega(n)\ge 5 \log\log N\}\big| &\le \frac{N\big(\sum_{q\le N}\frac{1}{q}\big)^k}{k!}\le N\bigg(\frac{e}{k}\bigg)^{k}\cdot \bigg(\sum_{q\le N}\frac{1}{q}\bigg)^k \ll  N (\log N)^{-3}.
\end{align*}
We have used that $\sum_{q\le N}\frac{1}{q}\le \log\log N + O(1)$ by \cref{thm:Merten} and that $k!\ge (k/e)^k$. 
\end{proof}

We next require an estimate that an integer in $[1, N]$ has a prime power divisor larger than $t$. This will be used to ensure that our sets of integers contain only smooth numbers.
\begin{lemma}\label{lem:smooth}
Let $N$ be sufficiently large. Given $t\in [2, N^{1/4}]$, define $Y$ be the set of integers divisible by a prime power larger than $N/t$. Then 
\[|Y\cap [1, N]|\le \frac{2N\log t}{\log N}.\]
\end{lemma}
\begin{proof}
We have that 
\begin{align*}
|Y|&\le \sum_{t\le q\le N}\frac{N}{q}\le  N\sum_{N/t\le p\le N}\frac{1}{p}+ N^{3/4} \le N\bigg(\log\bigg(\frac{\log N}{\log (N/t)}\bigg)\bigg) + O(N/(\log N)^{2})\\
&\le N \log\bigg(1 + \frac{3\log t}{2\log N}\bigg) \le \frac{2N\log t}{\log N};
\end{align*}
here we have used \cref{thm:Merten}.
\end{proof}

We finally need the following elementary consequence of the Fundamental Lemma of Sieve Theory. Recall that for a set $I$ that $I_p = \{n \in I : p | n\}$.
\begin{lemma}\label{lem:fund-lem}
Let $I$ be an interval of integers of length $X$. Let $\mc{P}$ be a set of primes, each of size at most $\exp((\log X)/(\log\log X)^{1/2})$. We have that 
\[\bigg|I\setminus \cup_{p\in \mc{P}}I_p\bigg|\asymp X \cdot \prod_{p\in \mc{P}}\bigg(1-\frac{1}{p}\bigg) \asymp X \cdot \exp(-R(\mc{P})) .\]
\end{lemma}
\begin{proof}
We apply the statement of the Fundamental Lemma of Sieve Theory as given in \cite[Theorem~18.11(b)]{Kuo19} with $D = X^{1/2}$, $m=1$, and $\kappa=1$. Axiom 1 there holds with $\nu(d) = 1$ for $d|\prod_{p\in \mc{P}} p$ with $|r_d|\le 2$. Axiom 2 holds with $\kappa = 1$ and $C$ being an absolute constant. Axiom 3 holds with $D = X^{1/2}$ and noting that $\sum_{n\le X^{1/2}}\tau(n) \ll X^{1/2+o(1)}$ by the divisor bound; here $\tau(n)$ is the divisor function as standard. 
\end{proof}

\subsection{Remaining preliminaries}\label{subsec:remain}
We next require the following standard estimates regarding exponential phases. 
\begin{fact}\label{fact:Taylor}
Let $x\in [-1/2,1/2]$ and $q\in [0,1]$. Then 
\begin{align*}
\big|(1-q) + qe(x) - e(qx)(1 - 2\pi q(1-q)x^2)\big| &\ll x^3\\
\big|(1-q) + qe(x)\big| &\le 1 - 8q(1-q)x^2.
\end{align*}
\end{fact}
\begin{proof}
Notice that 
\begin{align*}
\big|(1-q) + qe(x) - e(qx)(1 - 2\pi q(1-q) x^2)\big| &= \big|(1-q)e(-qx) + qe((1-q)x) - (1 - 2\pi q(1-q) x^2)\big|\\
&\le O(x^3).
\end{align*}
We have used the Taylor expansion
\begin{align*}
(1-q)e(-qx) &= (1-q)(1 + 2\pi iqx - 2\pi^2 q^2x^2)+ O(x^3))\\
qe((1-q)x) &= q(1 + 2\pi i (1-q)x  - 2\pi^2 (1-q)^2x^2) + O(x^3).
\end{align*}
For the second part, one can easily verify the numerical inequality that $\cos(2\pi z)\le 1 - 8z^2$ for $|z|\le 1/2$. Thus we have 
\begin{align*}
\big|(1-q) + qe(x)\big|^2 &= (1-q)^2 + q^2 + 2q(1-q)\cos(2\pi x)\\
&\le 1 - 16 q(1-q)x^2.
\end{align*}
As $(1-x)^{1/2}\le 1-x/2$, we have the desired estimate. 
\end{proof}

We also require the Azuma--Hoeffding inequality (see \cite[Theorem~2.25]{JLR00}).
\begin{lemma}[Azuma--Hoeffding inequality]\label{lem:azuma}
Let $X_0, \ldots, X_n$ form a martingale sequence such that $|X_k-X_{k-1}|\le c_k$ almost surely. Then 
\[\mb{P}[|X_0-X_n|\ge t]\le 2\exp\bigg(-\frac{t^2}{2\sum_{k=1}^nc_k^2}\bigg).\]
\end{lemma} 
\begin{remark}
We refer to $\sum_{k=1}^{n} c_k^2$ as the ``variance proxy''.
\end{remark}

\section{Proof of simplified main proposition}
We first give a simplified version of the main proposition \cref{prop:main}. In particular, in this setting we may assume the set under consideration is the set of all integers $[M,N]$ which are sufficiently smooth and have few primes factors. This additional structure simplifies the proof and allows for slightly better quantitative dependences in \cref{thm:294}.

We first state a major arc lemma which will be used repeatedly; this a standard Fourier analytic computation.  
\begin{lemma}\label{lem:major-arc}
Let $N$ be sufficiently large, $N^{.95}\le M\le N$, $A\subseteq [M,N]$, and  $|A|\ge N^{.95}$. Suppose that $p_n\in[(\log N)^{-2}, 1-(\log N)^{-2}]$ for $n\in A$ and take $x/Q = \sum_{n\in A} p_n/n$. Then 
\[\frac{1}{Q} \sum_{|h|\le M/2} \Re\left(e\left(-\frac{hx}{Q}\right) \prod_{n \in A} (1-p_n+p_n e(h/n))\right) \ge \frac{3}{4Q}.\]
\end{lemma}
\begin{proof}
We break into cases based on the size of $h$. We first consider $|h| \le M^{3/5}$. Via the first item of \cref{fact:Taylor}, we have
\begin{align*}
\frac{1}{Q} &\sum_{|h| \le M^{3/5}} e\left(-\frac{hx}{Q}\right) \prod_{n \in A} (1-p_n+p_n e(h/n))\\
& = \frac{1}{Q} \sum_{|h| \le M^{3/5}} e\left(-\frac{hx}{Q}\right) \prod_{n \in A} \exp(p_n h/n)\bigg(1 - \frac{2\pi p_n(1-p_n)h^2}{n^2} + O\bigg(\frac{h^3}{n^3}\bigg)\bigg)\\
& = \frac{1}{Q} \sum_{|h| \le M^{3/5}}  \prod_{n \in A} \exp\bigg( - \frac{2\pi p_n(1-p_n)h^2}{n^2}\bigg)\bigg(1 + O\bigg(\frac{h^3}{n^3}\bigg)\bigg) \\
&= \frac{1}{Q} \sum_{|h| \le M^{3/5}} (1 + O(M^{-1/5}))\exp\bigg(-\sum_{n\in A}\frac{2\pi p_n(1-p_n)h^2}{n^2}\bigg).
\end{align*}
Here $O(z)$ denotes a complex number with norm bounded by $z$. We therefore immediately find that 
\begin{align*}
&\frac{1}{Q} \sum_{|h|\le M^{3/5}} \Re\left(e\left(-\frac{hx}{Q}\right) \prod_{n \in A} (1-p_n+p_n e(h/n))\right) \ge  \frac{5}{6Q} \sum_{|h| \le M^{3/5}}\exp\bigg(-\sum_{n\in A}\frac{2\pi p_n(1-p_n)h^2}{n^2}\bigg)\ge \frac{5}{6Q}
\end{align*}
where the final line relies solely on the value at $h = 0$. We now argue that the contribution of terms with $M^{3/5} \le |h| \le M/2$ is negligible. Via the second item of \cref{fact:Taylor} (using that $|h/n|\le M/(2M)\le 1/2$), we have
\begin{align*}
\bigg|\frac{1}{Q} &\sum_{|h|\in [M^{3/5},M/2]} e\left(-\frac{hx}{Q}\right) \prod_{n \in A} (1-p_n+p_n e(h/n))\bigg| \le \frac{1}{Q}\sum_{|h|\in [M^{3/5},M/2]}\prod_{n \in A}\bigg|(1-p_n+p_n e(h/n))\bigg|\\
&\le \frac{1}{Q}\sum_{|h|\in [M^{3/5},M/2]}\prod_{n \in A}\bigg(1 - \frac{8p_n(1-p_n)h^2}{n^2}\bigg)\le \frac{M}{Q}\cdot \exp\bigg(\frac{-8p_n(1-p_n)|A|M^{6/5}}{N^2}\bigg).
\end{align*}
As $M\ge N^{.95}$, $p_n \in[(\log N)^{-2},1-(\log N)^{-2}]$, and $|A| \ge N^{.95}$, we see that the above is trivially $\ll 1/N^2$. 
\end{proof}

\begin{proposition}
\label{prop:simple}
There exists an absolute constant $C\ge 1$ such that the following holds. Let $N$ be sufficiently large. Consider parameters $N^{.9999} \le S \le K \le M \le N/10$ satisfying:
\[S \le \min\bigg(\frac{M^2}{CN}, \frac{K^3}{CN^2(\log \log N)^{5}}\bigg), \text{ and }N(\log N)^{-10} \le K \le 10^{-7}N(\log N)^{-1},\]
Let $A = \{n \in [M, N] : n \text{ is } S\mathrm{-smooth}, \wt{\Omega}(n) \le 5 \log \log N, \Omega(N)\le 10\log\log N\}$, and let $(\log \log N)^{-1} \le p_M, \dots, p_N \le 1/2$. Then $B$ be sampled where $n \in A$ is contained in $B$ with probability $p_n$, independently. Let $\mc{Q}_A = \{ q \le S \}$ and $Q = [\mc{Q}_A]$. If $x \in [1, Q]$ is an integer such that $\mb{E}[R(B)] = x/Q$, then $\mb{P}[R(B) = x/Q] \ge 1/(4Q)$.
\end{proposition}
Before giving the proof, we establish a lower bound on the size of $A$.
\begin{lemma}\label{lem:A}
Let $N$ be sufficiently large and $A$ as in \cref{prop:simple}. Then $|A| \ge .89 N$.
\end{lemma}
\begin{proof}
This follows from \cref{lem:smooth,lem:divisor}.
\end{proof}

Now we prove \cref{prop:simple}. Its proof is broadly based on \cite[Proposition 2]{Blo21}, except that we handle the minor arcs in quite a different manner in this simpler setting.
\begin{proof}[Proof of \cref{prop:simple}]
Recall the following identity for integers $a, b \neq 0$:
\[ 1_{a/b \in \mb{Z}} = \frac{1}{b} \sum_{-b/2 < h \le b/2} e\left(\frac{ha}{b} \right). \]
We can write:
\begin{align}
\mb{P}[R(B) - x/Q \in \mb{Z}] &= \sum_{B \subseteq A} \prod_{n \in B} p_n \prod_{n \in A \setminus B} (1-p_n) \cdot \frac{1}{Q} \sum_{-Q/2 < h \le Q/2} e\left(\sum_{n \in B} \frac{h}{n} - \frac{hx}{Q} \right) \nonumber \\
&= \frac{1}{Q} \sum_{-Q/2 < h \le Q/2} e\left(-\frac{hx}{Q}\right) \prod_{n \in A} (1-p_n+p_n e(h/n)) \nonumber \\
&= \frac{1}{Q} \sum_{-Q/2 < h \le Q/2} \Re\left(e\left(-\frac{hx}{Q}\right) \prod_{n \in A} (1-p_n+p_n e(h/n))\right), \label{eq:expressionsimple}
\end{align}
The final line holds because the expression is a probability (or alternatively by noting that the expression conjugates upon negating $h$).

Our ultimate goal is to establish that \eqref{eq:expressionsimple} is at least $\frac{1}{2Q}$. Before proving this, let us see why the proposition then follows. It suffices to prove that $\mb{P}[|R(B) - x/Q| \ge 1] < \frac{1}{4Q}$. Because $\mb{E}[R(B)] = x/Q$, the previous probability is at most $2\exp(-M^2/(2N))$ by Azuma-Hoeffding (\cref{lem:azuma}). By \cref{thm:Merten}, we have $Q \le \prod_{q \le S} \ll e^{5S}$. For our choice of $S, M$, it holds that $2\exp(-M^2/(2N)) \le e^{-6S}$ for $N\gg 1$.

The remainder of the proof is devoted to estimating the expression in \eqref{eq:expressionsimple}. Via \cref{lem:major-arc} (applying \cref{lem:A} to lower bound the size of $A$), we find that 
\[\frac{1}{Q} \sum_{|h| \le M/2} \Re\left(e\left(-\frac{hx}{Q}\right) \prod_{n \in A} (1-p_n+p_n e(h/n))\right)\ge \frac{3}{4Q}.\]

The remainder of the analysis consider $h$ such that $|h|>M/2$; e.g. the minor arcs. 

\textbf{Minor arcs, upper bound:} For each $n$, let $h_n$ be the unique integer in $(-n/2, n/2]$ with $h \equiv h_n \imod{n}$. Let 
\[t = \frac{100N^2\log N (\log \log N)^2}{K^2}.\] Finally, let $I_h = (h - K/2, h + K/2)$ be an interval and define
\[ \mc{D}_h = \{ q \in \mc{Q}_A : |\{ n \in A_q : h_n \ge K/2 \}| < t\}.\] 

We now establish a relationship between the Fourier coefficient at $h$ and the size of $|\mc{Q}_A \setminus \mc{D}_h|$. Recall that each $n\in A$ satisfies $\Omega(n)\le 10 \log\log N$. Therefore using the second item of \cref{fact:Taylor}, we have 
\begin{align*}
\left(\prod_{n \in A} |(1-p_n+p_n e(h/n))| \right)^{10 \log \log N} &\le \prod_{q \in \mc{Q}_A} \prod_{n \in A_q} |(1-p_n+p_n e(h/n))| \\
&\le \prod_{q \in \mc{Q}_A \setminus \mc{D}_h} \exp\bigg(-p_n \cdot \frac{K^2}{N^2} \cdot t\bigg)\\
& \le \exp(-100|\mc{Q}_A \setminus \mc{D}_h| \log N \log \log N)
\end{align*}
because $p_n \ge (\log \log N)^{-1}$, and thus 
\begin{equation}\label{eq:minor-bound-simple}
\prod_{n \in A} |(1-p_n+p_n e(h/n))| \le N^{-10|\mc{Q}_A \setminus \mc{D}_h|}.
\end{equation}

\textbf{Minor arcs, establishing divisibility:}
The crucial step in our proof is establishing there exists an element $x\in I_h$ such that $[\mc{D}_h]| x$.

For $q \in \mc{D}_h$, define 
\[T_q = |\{ n \in A_q : |h_n| < K/2 \}|.\]
By definition of $\mc{D}_h$, we have $|A_q \setminus T_q| < t$. Our next goal is to find a prime $p' \ll t \cdot \frac{(\log \log N)^2}{\log N}$ such that $\gcd(p, q) = 1$ and 
\[ |A_{pq} \setminus T_q| \le \frac{\log N}{1000 (\log \log N)^2}. \]
This follows by an averaging argument. Because every $n \in A$ satisfies $\Omega(n) \le 10 \log \log N$,
\[ 10 \log \log N |A_q \setminus T_q| \ge \sum_{p' \le 10^6t \cdot \frac{(\log \log N)^3}{\log N}} |A_{pq} \setminus T_q|. \]
The desired $p'$ exists because there are least $1000t \frac{(\log \log N)^2}{\log N}$ primes $p' \le 10^6t \cdot \frac{(\log \log N)^3}{\log N}$, by the definition of $t$ and $K \ge N(\log N)^{-10}$. Furthermore note that 
\[qp'\ll S\cdot t (\log\log N)^3/(\log N) \ll S \cdot \frac{N^2 (\log\log N)^5}{K^2}\le K;\]
the final inequality holds provided $C$ is sufficiently large in the original assumption on parameters. 

Note that $qp' \le K$. Our next goal is to find $r$ which is the product of at most two distinct primes at most $S$, with $qp'r \in [2K, 100K]$. If $S \ge 100K(qp')^{-1}$, then choose $r$ to be a prime with $\gcd(r, qp') = 1$ in $[2K(qp')^{-1}, 100K(qp')^{-1}]$. Otherwise, let $r_1$ be a prime in $[S/2, S]$ with $\gcd(r_1, qp') = 1$, and if $qp'r_1 < 2K$ still, then let $r_2$ be a prime in $[2K(qp'r_1)^{-1}, 100K(qp'r_1)^{-1}]$ with $\gcd(qp'r_1, r_2) = 1$. Then let $r = r_1r_2$. Such $r_1, r_2$ exist because $S^2 \ge N^{1.98}$ which is much larger than $K$. Note that clearly, 
\[ |A_{qp'r} \setminus T_q| \le |A_{qp'} \setminus T_q| \le \frac{\log N}{1000 (\log \log N)^2}. \]

Let $\mc{P} = \{20 \log N \le p \le 40 \log N\}$. Define
\[ \mc{P}_q = \{p \in \mc{P} : \gcd(p, qp'r) = 1, A_{qp'rp} \subseteq T_q \}. \]
We argue that 
\[ |\mc{P}_q| \ge |\mc{P}| - \frac{\log N}{15\log \log N} - 4 \ge \frac{9}{10}|\mc{P}|. \] Because $\omega(n) \le \Omega(n) \le 10 \log \log N$, for each $n \in A_{qp'r} \setminus T_q$, there are at most $10 \log \log N$ different $p$ with $n \in A_{qp'rp} \setminus T_q$, which excludes at most
\[ 10 \log \log N|A_{qp'r} \setminus T_q| \le \frac{\log N}{15 \log \log N} \] primes. The condition $\gcd(p, qp'r) = 1$ excludes at most $4$ more primes.

Note that $qp'r \ge K$, and thus there is a unique multiple $x_q$ of $qp'r$ in $I_h$, if it exists. Let $p \in \mc{P}_q$. Because $K \le qp'rp \le 4000K \log N \le N/2000$, there is a multiple of $n$ of $qp'rp$ that is $S$-smooth in the range $[N/2, N]$. We also require the multiple to have sufficiently few prime divisors; here we use that $\wt{\Omega}(n)\le 5\log\log N$ implies that $\Omega(qp'rp)\le 5\log\log N + 4$. Because $n \in A_{qp'rp} \subseteq T_q$, we know that $|h_n| < K/2$, there is $x \in I_h$ with $n | x$. Because $qp'r | x$, we know that $x = x_q$, and thus $p | x$ too.

Now, for any $q_1, q_2 \in \mc{D}_h$, we know that $|\mc{P}_{q_1} \cap \mc{P}_{q_2}| \ge .8|\mc{P}|$. By the reasoning in the above paragraph, we deduce that $\prod_{p \in \mc{P}_{q_1} \cap \mc{P}_{q_2}} p | x_{q_1}, x_{q_2}$, and thus $x_{q_1} = x_{q_2}$ because
\[ \prod_{p \in \mc{P}_{q_1} \cap \mc{P}_{q_2}} p > (20 \log N/(\log\log N))^{5\log N/(\log\log N)} > N. \] Thus, all $x_q$ are equal for $q \in \mc{D}_h$, and thus $[\mc{D}_h]$ divides the common $x_q$.

\textbf{Minor arcs, finishing the proof:} 
We are now in position to complete the proof. Consider all $h\in (-Q/2,Q/2]$ such that $\mc{D}_h = \mc{D}$. As $[\mc{D}_h] | x$ for an element $x\in I_h$, we have that the number of such $h$ is bounded by 
\begin{equation}\label{eq:fin-bound-simple}
(K + 1) \cdot \frac{[\mc{Q}_A]}{[\mc{D}]}\le N \cdot \prod_{q \in \mc{Q}_A \setminus \mc{D}} q \le N^{|\mc{Q}_A \setminus \mc{D}|+1}. 
\end{equation}
As we have restricted attention to $|h|>M/2\ge K/2$, we have that $\mc{D}_h \neq \mc{Q}_A$. Therefore the total contribution over minor arcs, using \cref{eq:minor-bound-simple,eq:fin-bound-simple} and is bounded by 
\[\frac{1}{Q}\sum_{\mc{D} \subsetneq \mc{Q}_A}N^{|\mc{Q}_A \setminus \mc{D}|+1} \cdot N^{-10|\mc{Q}_A \setminus \mc{D}|} \le \frac{1}{Q}\cdot \sum_{s\ge 1} N^{s+1} \cdot N^{s} \cdot N^{-10s} \le 2/(QN).\]
This completes the proof. 
\end{proof}

\section{Proof of \cref{thm:294,thm:297,thm:305}}\label{sec:294-297}
We now detail the proof of \cref{thm:294,thm:297,thm:305}. We first prove \cref{thm:297}. 
\begin{proof}[{Proof of \cref{thm:297}}]
\textbf{Upper bound:} Let $\mc{A} = \{ A \subseteq [1, N] : R(A) \le 1\}.$ For any $\lambda > 0$,
\begin{align*}
|\mc{A}|e^{-\lambda N} &\le \sum_{S\subseteq[1, N]}e^{-\lambda N\sum_{i\in S}\frac{1}{i}} = \prod_{i=1}^N\big(1 + e^{-\lambda N/i}\big)\\
& = \exp\bigg(N\sum_{i=1}^{N}\frac{1}{N}\log\big(1+e^{-\lambda/(i/N)}\big)\bigg)= \exp\bigg(N\int_{0}^{1}\log\big(1 + e^{-\lambda/x}\big)~dx + o(N)\bigg),
\end{align*}
where the final step follows by Euler-Maclaurin. Thus
\[\frac{1}{N}\log|\mc{A}|\le\lambda+\int_0^1\log\big(1+e^{-\lambda/x}\big)~dx+o(1)\]
Plugging in $\lambda=\lambda^\ast$ gives the stated result.

\textbf{Lower bound:} For $i\in [1, N]$, define $p_i = \frac{e^{-\lambda^\ast N/i}}{1 + e^{-\lambda^\ast N/i}}$ where $\lambda^\ast$ as defined in \cref{thm:297}.

We apply \cref{prop:simple} with $M = N(\log\log\log N)^{-1/2}$ and $S = N/(\log N)^{4}$. Under these conditions, via \cref{lem:smooth,lem:divisor} we find that 
\[|[M,N]\setminus A|\ll N \log\log N (\log N)^{-1}.\]
We therefore see that 
\[\sum_{i\in A} \frac{p_i}{i} = (1\pm o(1)).\]
Let $\wt{p}_i$ be $p_i$ scaled by a uniform $(1-o(1))$ factor, such that
\[\sum_{i\in A}\frac{\wt{p}_i}{i} = 1\text{ and }(\log\log N)^{-1}\le \wt{p}_i\le 1/2.\]

Let $\mu$ denote the measure of $A$ where every element $i$ is chosen with probability $\wt{p}_i$. By \cref{prop:simple}, we have that 
\[\mb{P}_{B\sim \mu}[R(B) = 1] \ge \exp(-CN/(\log N)^4).\]
Note that for each $B$ with $R(B) = 1$ that the probability mass of choosing $B$ is:
\begin{align*}
\prod_{i \in A} (1 - \wt{p}_i) \prod_{i \in B} \frac{\wt{p}_i}{1-\wt{p}_i} &\le \exp(o(N)) \prod_{i \in A} (1 - p_i) \prod_{i \in B} \frac{p_i}{1-p_i} \\
&\le \exp(o(N)) \prod_{i \in [1, N]} (1 - p_i) \exp(-\lambda^\ast n R(B)) \le \exp(-\gamma^\ast N+ o(N));
\end{align*}
rearranging and using the probability lower bound on $\mb{P}_{B\sim \mu}[R(B) = 1]$ completes the proof. 
\end{proof}

We now prove \cref{thm:294,thm:305}; the proofs are quite similar. Towards this, we first show how to obtain fractions with denominators which are smooth, using an interval of constant width.
\begin{lemma}
\label{lem:294}
Let $N$ be sufficiently large, $t$ be $S$-smooth for $S = cN(\log N)^{-3}(\log \log N)^{-4}$, and $t/3 \le s \le t$. There is a subset $A \subseteq [N/16, N]$ with $R(A) = s/t$.
\end{lemma}
\begin{proof}
Apply \cref{prop:simple} with the following choices of parameters: \[ M = N/16, K = 10^{-7}N(\log N)^{-1}, \enspace \text{ and } \enspace S = cN(\log N)^{-3}(\log \log N)^{-4}, \] for sufficiently small $c$. These satisfy the hypotheses of \cref{prop:simple}. By \cref{lem:A} we know that $|A| \ge .89N$, so $R(A) \ge \sum_{n=.11N}^N \frac{1}{n} \ge 2$. Also, $R(A) \le 3$. Let $p_n = \frac{s/t}{R(A)}$ for all $n = N/16, \dots, N$ so that $\mb{E}[R(B)] = s/t$. Finally, it suffices to check that there is an integer $x$ with $x/Q = s/t$, i.e., $t | Q$. This holds exactly because $t$ is $S$-smooth.
\end{proof}
We will now prove \cref{thm:294} by applying \cref{lem:294} a few times.
\begin{proof}[{Proof of \cref{thm:294}}]
We first handle the upper bound. 

\textbf{Proving that $t(N)\ll N/(\log N)$:} Let $t$ be a prime larger than $10N(\log N)^{-1}$. Let $n_2, \dots, n_k > t$ be distinct multiples of $t$, and $m_1, \dots, m_{k'}$ be not multiples of $t$ such that \[ \frac{1}{t} + \sum_{j=2}^k \frac{1}{n_j} + \sum_{j=1}^{k'} \frac{1}{m_j} = 1. \]
Multiplying both sides by $t$ and taking modulo $t$ gives:
\[ 1 + \sum_{j=2}^k \frac{1}{n_j/t} \equiv 0 \imod{t}. \]
If $t > 10N(\log N)^{-1}$, then
\[ \mathrm{lcm}(n_2/t, \dots, n_k/t) < \mathrm{lcm}(1, 2, \dots, (\log N)/10), \] by \cref{thm:Merten}. Thus, writing $a/b = \sum_{j=2}^k \frac{1}{n_j/t}$ we know that $a \le N^{1/2}, b \le N^{1/3}$. Clearly, $1 + a/b = (a+b)/b \neq 0 \imod{t}$ because $t > a+b$.

\textbf{Proving the lower bound on $t(N)$:} We break into cases based on $t$. If $t = 1$ the proof is trivial. If $t \in[2, N^{9/10}]$, then apply \cref{lem:294} for $s = t-1$.

Let $\wt{N} = \min(N/(100T), (\log N)^2)$. Apply \cref{prop:simple} for the following parameter choices: $M = \wt{N}/16$, $K = 10^{-7}\wt{N}(\log \wt{N})^{-1}$, and $S = c\wt{N}(\log \wt{N})^{-3}(\log \log \wt{N})^{-4}$, for sufficiently small $c$. Let $Q$ be as defined in \cref{prop:simple}. By \cref{thm:Merten}, we know that $Q \gg 2^S$. If
\[ t \le \frac{N}{\log N(\log \log N)^3 (\log \log \log N)^{O(1)}}, \] then
\[ \tilde{N} \ge \log N (\log \log N)^3 (\log \log \log N)^{O(1)}, \] and hence $S \ge 10 \log N$. Thus, $2^S > 100t$ and $Q \ge 100t$.

Let $s = (mt-Q)$ an integer $m$ chosen so that $mt-Q \in [Q/3, 2Q/3]$; such an $m$ exists because $Q > 100t$. Applying \cref{lem:294} to the fraction $s/Q$ proves that there is a subset $B \subseteq [\wt{N}/15, \wt{N}]$ with $R(B) = s/Q$. In the notation of \cref{thm:294}, we let $n_2, n_3, \dots, n_{|B|+1}$ be $tn$ for $n \in B$. Then, for $n_1 = t$, we have:
\[ \frac{1}{n_1} + \frac{1}{n_2} + \dots + \frac{1}{n_{|B|+1}} = \frac{1}{t}(1 + R(B)) = m/Q. \]
Note that $n_{|B|+1} \le \tilde{N}t < N/100$ for our choice of $\tilde{N}$.

Note that $mt-Q \le Q$, so $m/Q \le 2/t < 1/3$. Now, apply \cref{lem:294} again for $s = (Q-m)$, $t = Q$, and $M = N/16$. Because $Q$ is $S$-smooth for $S \le \tilde{N} \le (\log N)^2$, we obtain a subset $B' \subseteq [N/16, N]$ with $R(B') = (Q-m)/Q$. Appending this to the previous entries $n_2, \dots, n_{|B|+1}$ completes the construction.
\end{proof}

The proof of \cref{thm:305} is quite similar to the proof of \cref{thm:294}.
\begin{proof}[Proof of \cref{thm:305}]
Let $\wt{N} = (\log b)(\log\log b)^3(\log\log\log b)^{O(1)}$, and 
\[ M = \wt{N}/16, K = 10^{-7}\wt{N}(\log \wt{N})^{-1}, \enspace \text{ and } \enspace S = c\wt{N}(\log \wt{N})^{-3}(\log \log \wt{N})^{-4}. \]
Let $\mc{Q}_A = \{q \le S \}$ and $Q = [\mc{Q}_A] \gg 2^S > 10b$ (\cref{thm:Merten}), by our choice of $S$.
Let $x$ be an integer such that $Q/3 \le aQ - xb \le 2Q/3$. Such $x$ exists because $a \ge 1$ and $b \le Q/10$. Thus, by \cref{lem:294} for $s = aQ-xb, t = Q$, we can find $A \subseteq [M, \wt{N}]$ with $R(A) = s/t$. Let $\wt{A} = \{ bn : n \in A\}$. Note that all elements of $\wt{A}$ are at most $b\tilde{N}$ and 
\[ \frac{a}{b} - R(\wt{A}) = \frac{a}{b} - \frac{1}{b} R(A) = \frac{a}{b} - \frac{1}{b} \frac{aQ - xb}{Q} = \frac{x}{Q}. \]
Now, we will express $x/Q$ as the sum of unit fractions. To start, note that $Q/3 \le aQ - xb \le 2Q/3$ implies that
\[ aQ \ge \left(a - \frac13\right)Q \ge xb \ge \left(a - \frac23\right)Q \ge \frac13 aQ, \]
so $\frac{a}{b} \ge x/Q \ge \frac{1}{3}\frac{a}{b}$. Let $y \in [\lceil Q/(3x) \rceil, \lfloor Q/x \rfloor]$ be an arbitrary integer. By \cref{lem:294}, there is a subset $B \subseteq [\tilde{N}/16, \tilde{N}]$ such that $R(B) = yx/Q$. Thus for $\wt{B} = \{yn : n \in N\}$, we know that $R(\wt{B}) = x/Q$. Thus as long as $\wt{A}$ is disjoint from $\wt{B}$, we know that $R(\wt{A} \cup \wt{B}) = a/b$ as desired.

To choose $y$ to make $\wt{B}$ disjoint from $\wt{A}$, we split into two cases based on the size of $a$. If $a > 16$, then $y \le Q/x \le b/a$. Then $\max(\wt{B}) \le y\tilde{N} < b\tilde{N}/16 \le \min(\wt{A})$. If $a \le 16$, then there is a prime number $y \in [\lceil Q/(3x) \rceil, \lfloor Q/x \rfloor]$ which is not a divisor of $b$ (recall that $Q/x \ge b/a \ge b/16$). Then every element of $\wt{B}$ is a multiple of $y$, while $\wt{A}$ only contains numbers of the form $bn$ for $n \le \wt{N}$. So $\wt{A}$ is disjoint from $\wt{B}$ for the choice of $y$.
\end{proof}

\section{Proof of main proposition}

We first require that given a set $A$ and an integer $q$ where $A_{q}$ is large (in terms of reciprocal sum), we may find an integer $d$ such that $R(A_{qd})$ is sufficiently large and $qd$ is in an appropriate range. This is a slight simplification of \cite[Lemma~5.1]{Blo21}.
\begin{lemma}
\label{lem:lift}
There is an absolute constant $C\ge 1$ such that the following holds.

Suppose that $N$ is sufficiently large and $\delta\in [0,1/2]$, and furthermore that $A$ and $q$ are such that:
\begin{itemize}
    \item $A\subseteq [M,N]$ with $N^{.99} \le M \le N/10$,
    \item $q$ is a prime power with $q\le M \exp(-(\log N)^{1-\delta})$ and $q\cdot R(A_q) \ge \eta>0$,
    \item $\max_{n\in A}\Omega(n) \le 5\log \log n$.
\end{itemize}

Let $H = \exp(\eta (\log N)^{1-\delta}(\log \log N)^{-3} (\log(N/M))^{-1})$. There exists an integer $d$ and a subset $A_{qd}^\ast\subseteq A_{qd}$ such that:
\begin{itemize}
    \item $\min(A_{qd}^\ast)\ge H \cdot qd$,
    \item $qd\ge M \exp(-(\log N)^{1-\delta})$,
    \item $qd \cdot R(A_{qd}^\ast) \ge  C^{-1} \eta ((\log N)^{\delta} \log \log N)^{-1}$.
\end{itemize}
\end{lemma}
\begin{proof}
Let $y = \exp((\log N)^{1-\delta}/(10 \log \log N))$. For $n \in A_q$, define
\[d_n = \prod_{\substack{p | (n/q)\\ p > y}} p^{v_p(n)}.\] 
Note that for all $n \in A_q$
\[qd_n \ge n/y^{\Omega(n)} \ge M/\exp((\log N)^{1-\delta}).\] This implies the lower bound of the first item.

Consider the set of all integers with less than two prime factors in the range $[H,y]$; call such numbers ``poor'' and let $A_q'$ be the set of poor numbers in $A_q$. Note that $M/q\ge \exp((\log N)^{1-\delta})$ and therefore applying \cref{lem:fund-lem} we find that 
\begin{align*}
q \cdot R(A_q') &\le \sum_{\substack{n\in [M/q,N/q]\\n\text{ is poor}}}\frac{1}{n} \ll \log\bigg(\frac{N}{M}\bigg) \cdot \frac{\log H}{\log y}(\log\log N)\le \frac{\eta}{2}.
\end{align*}
Here, we have applied \cref{lem:fund-lem} in the following way: the number of integers in $[X, 2X]$ (for $X \ge M/q$) that are divisible by some $p \in [H, y]$ and no other primes in $[H, y]$ is at most $\ll \frac{X}{p} \frac{\log H}{\log y}$ (one can check that $M/(qp) \ge M/(qy)$ is sufficiently large in terms of $y$ to apply \cref{lem:fund-lem}). Now, $\sum_p \frac{X}{p} \frac{\log H}{\log y} \ll X \cdot \frac{\log H}{\log y}(\log \log N)$.

Let $D$ be the set of integers with all prime factors in $[y, N]$ and define $\wt{A_q} = A_q\setminus A_q'$. Note that $d_n \in D$ for all $n \in A_q$, and $q \cdot R(\wt{A}_q) \ge q \cdot R(A_q) - q \cdot R(A_q') \ge \eta/2$. Furthermore for $n\in \wt{A_q}$, we have that $qd_n\le n/H$ as we are able to remove at least $1$ primes factor in $[H,y]$. Furthermore note that
\begin{align*}
    \frac{\eta}{2} &\le q \cdot R(\wt{A_q}) = \sum_{n \in \wt{A_q}} \frac{q}{n} \le \sum_{d \in D} \frac{1}{d} \sum_{\substack{n\in \wt{A_q}\\q | n, d_n = d}} \frac{qd}{n} \le \sum_{\substack{d \in D\\ qd\ge M/\exp((\log N)^{1-\delta}}}\frac{1}{d} \cdot qd \cdot R(A_{qd}^\ast).
\end{align*}
where $A_{qd}^\ast := \{n \in \wt{A_q} : d_n = d \}$.
Take $d$ such that $qd \cdot R(A_{qd}^\ast)$ is maximal. The result follows from the bound
\[ \sum_{d \in D} \frac{1}{d} \le \prod_{p \in [y, N]} (1 + 1/p + 1/p^2 + \dots) \le \prod_{p \in [y, N]} \frac{p}{p-1} \ll \frac{\log N}{\log y} \ll (\log N)^{\delta} \log \log N.\]
\end{proof}

The following is the main proposition we will use to prove \cref{thm:main}. Its proof is broadly based on the ideas of \cite{Cro03}, \cite[Proposition 2]{Blo21}, and \cite[Proposition 3]{Blo21}. The main difference is that we use a sieve via \cref{lem:fund-lem} to estimate the density of exceptional primes, as opposed to something like \cite[Lemma 4]{Blo21}.
\begin{proposition}\label{prop:main}
There exists an absolute constant $C\ge 1$ such that the following holds. Fix $\delta \in (0,1), \eps\in (0,1/10)$ and suppose $N$ be sufficiently large in terms of $\eps, \delta$. Fix parameters $\eta, S, K, M, N$ such that $\eta \ge (\log N)^{-1}$ and
\[N^{.99} \le S \le K \le M \le N/10^4.\] 
Suppose we additionally have:
\begin{align*}
    \Gamma &:= \max\bigg(\frac{\eta}{(\log N)^{\delta}(\log\log N)^3}, \frac{\eta^2(\log N)^{1-2\delta}}{\log(N/M)^2 (\log\log N)^5}\bigg)\\
    S &\le \min\left(\frac{M^2}{CN}, \frac{\eta MK^2}{N^2(\log N)^3}\right), \\
    K &\le M \exp(-(\log N)^{1-\delta}), \\
    C &\le \frac{\Gamma^2}{(\log N)^{2\eps}(\log(N/M) + (\log N)^{1-\delta})}.
\end{align*}

Let $A\subseteq [M,N]$ be such that
\begin{itemize}
    \item Every $n\in A$ is $S$--smooth,
    \item $\max_{n\in A}\Omega(n)\le 5\log\log N$,
    \item $\min_{q\in \mc{Q}_A} qR(A_q)\ge \eta$.
\end{itemize}
Furthermore let $x\in \mb{N}$ be such that $R(A)\in [(1 + 1/(\log N))x/Q, (\log N)x/Q]$. Then there exists $B \subseteq A$ such that $R(B) = x/Q$.
\end{proposition}
Note that in \cref{prop:main}, $\Gamma$ is the $\max$ of two terms. The former is used for \cref{thm:main}, and the latter is for \cref{thm:300,prop:310}.
\begin{proof}
As noted in the proof of \cref{prop:simple}, for integers $a, b \neq 0$:
\[ 1_{a/b \in \mb{Z}} = \frac{1}{b} \sum_{-b/2 < h \le b/2} e\left(\frac{ha}{b} \right). \]
Let $\tau = (x/Q)/R(A)$. Note that $\frac{1}{\log N} \le \tau \le \frac{\log N}{1+\log N}$. Consider generating a subset $B \subseteq A$ where each element $n \in A$ is included in $B$ with probability $\tau$. The probability that $R(B) - x/Q \in \mathbb{Z}$ is:
\begin{align}
\sum_{B \subseteq A} &\tau^{|B|}(1-\tau)^{|A \setminus B|} \frac{1}{Q} \sum_{-Q/2 < h \le Q/2} e\left(\sum_{n \in B} \frac{h}{n} - \frac{hx}{Q} \right)  \\
&=  \frac{1}{Q} \sum_{-Q/2 < h \le Q/2} e\left(-\frac{hx}{Q}\right) \prod_{n \in A} (1-\tau+\tau e(h/n)) \nonumber \\
&= \frac{1}{Q} \sum_{-Q/2 < h \le Q/2} \Re\left(e\left(-\frac{hx}{Q}\right) \prod_{n \in A} (1-\tau+\tau e(h/n))\right). \label{eq:expression}
\end{align}
The final line holds because the expression is a probability (also, by symmetry of $h$).

Our ultimate goal is to establish that \eqref{eq:expression} is at least $\frac{1}{2Q}$. Before proving this, let us see why the proposition then follows. It suffices to prove that $\mb{P}[|R(B) - x/Q| \ge 1] < \frac{1}{4Q}$. Because $\mb{E}[R(B)] = \tau R(A) = x/Q$, the previous probability is at most $2\exp(-M^2/(2N))$ by Azuma-Hoeffding (\cref{lem:azuma}). By \cref{thm:Merten}, we have $Q \le \prod_{q \le S} \ll e^{5S}$. For our choice of $S, M$, it holds that $2\exp(-M^2/(2N)) \le e^{-6S}$ for $N\gg 1$.

As $R(A)\ge \eta$, we trivially see that $|A|\ge N^{.95}$. Applying \cref{lem:major-arc}, where $p_i = \tau$ for all $i\in A$, we find that 
\[ \frac{1}{Q} \sum_{|h| \le M/2} \Re\left(e\left(-\frac{hx}{Q}\right) \prod_{n \in A} (1-\tau+\tau e(h/n))\right)\ge \frac{3}{4Q}.\]

The remainder of the proof is devoted to estimating the expression in \eqref{eq:expression} when $|h|>M/2$, e.g. the minor arcs. 

\textbf{Minor arcs, upper bound:} For each $n$, let $h_n$ be the unique integer in $(-n/2, n/2]$ with $h \equiv h_n \imod{n}$. Let 
\[t = \frac{50N^2\tau^{-1}\log N \log \log N}{K^2}.\] Finally, let $I_h = (h - K/2, h + K/2)$ be an interval and define
\[ \mc{D}_h = \{ q \in \mc{Q}_A : |\{ n \in A_q : h_n \ge K/2 \}| < t\}.\] 

We now establish a relationship between the Fourier coefficient at $h$ and the size of $|\mc{Q}_A \setminus \mc{D}_h|$. Recall that each $n\in A$ satisfies $\Omega(n)\le 5 \log\log N$. Therefore using the second item of \cref{fact:Taylor}, we have 
\begin{align*}
\left(\prod_{n \in A} |(1-\tau+\tau e(h/n))| \right)^{5 \log \log N} &\le \prod_{q \in \mc{Q}_A} \prod_{n \in A_q} |(1-\tau+\tau e(h/n))|\le \prod_{q \in \mc{Q}_A \setminus \mc{D}_h} \exp\bigg(-\tau \cdot \frac{K^2}{N^2} \cdot t\bigg)\\
& \le \exp(-50|\mc{Q}_A \setminus \mc{D}_h| \log N \log \log N)
\end{align*}
and thus 
\begin{equation}\label{eq:minor-bound}
\prod_{n \in A} |(1-\tau+\tau e(h/n))| \le N^{-10|\mc{Q}_A \setminus \mc{D}_h|}  
\end{equation}

\textbf{Minor arcs, establishing divisibility:}
The crucial step in our proof is establishing there exists an element $x\in I_h$ such that $[\mc{D}_h]| x$.

For $q \in \mc{D}_h$, define 
\[T_q = |\{ n \in A_q : |h_n| < K/2 \}|.\]
By definition of $\mc{D}_h$, we have $|A_q \setminus T_q| < t$. Note that 
\[R(T_q) \ge R(A_q) - \frac{t}{M} \ge \frac{\eta}{q} - \frac{t}{M}\ge \frac{\eta}{2q}.\]
Here we use that $\eta/(2q)\ge \eta/(2S) \ge t/M$ as $S\le \eta MK^2/(N^2 (\log N)^3)$.

We now apply \cref{lem:lift} with $\delta$ for each $q\in \mc{D}_h$ separately with the set $A = T_q$. Let $d_q$ be the multiple of $q$ output by \cref{lem:lift} and $T_q^{\ast}\subseteq (T_q)_{qd_q}$ be the set $A_{qd_q}^\ast$. The application is valid because $q \le S \le K \le M\exp(-(\log N)^{1-\delta})$, by the hypotheses. Note that $qd_q \ge M \exp(-(\log N)^{1-\delta})\ge K$; the final inequality is by assumption. Therefore there is at most one multiple of $qd_q$ in $I_h$ -- call it $x_q$. Next let 
\[\wt{T}_q = \{n/(qd_q) : n \in T_q^{\ast},~qd_q | n\}.\] By the second item of \cref{lem:lift}, we know that
\[ R(\wt{T}_q) \ge qd_q \sum_{\substack{n \in T_q^{\ast} \\ qd_q | n}} \frac{1}{n} \ge \frac{\eta}{C(\log N)^{\delta} \log \log N}.\]

Let $H = \exp(\eta (\log N)^{1-\delta}(\log \log N)^{-2} (\log(N/M))^{-1})$. Note that 
$\max(1,H)\le \min(\wt{T}_q)$ and $\max(\wt{T}_q)/ \min(\wt{T}_q)\le N/M$. Therefore 
\begin{align*}
\sum_{\min(\wt{T}_q)\le 2^{j}\le \max(\wt{T}_q)} &\frac{1}{j + 1} \ll \log\bigg(\frac{\log(\max(\wt{T}_q))}{\log(\min(\wt{T}_q))}\bigg)\ll \log\bigg(1 + \frac{\log(\max(\wt{T}_q)/\min(\wt{T}_q))}{\log(\min(\wt{T}_q))}\bigg) \\
&\ll \min\bigg(\log\log N, \frac{\log(N/M)}{\log(H)}\bigg) = \min\bigg(\log\log N, \frac{\log(N/M)^2(\log\log N)}{\eta(\log N)^{1-\delta}}\bigg) := L. 
\end{align*}

Note that 
\begin{align*}
\frac{\eta}{C(\log N)^{\delta} \log \log N} &\le R(\wt{T}_q) = \sum_{\min(\wt{T}_q)\le 2^{j}\le \max(\wt{T}_q)} R(\wt{T}_q \cap [2^{j}, 2^{j+1})) \\
&\ll \max_{\min(\wt{T}_q)\le 2^{j}\le \max(\wt{T}_q)} \left((j + 1) \cdot R(\wt{T}_q \cap [2^{j}, 2^{j+1}))\right) \cdot \sum_{\min(\wt{T}_q)\le 2^{j}\le \max(\wt{T}_q)} \frac{1}{j + 1} \\
&\ll \max_{\min(\wt{T}_q)\le 2^{j}\le \max(\wt{T}_q)} \left((j+1) \cdot R(\wt{T}_q \cap [2^{j}, 2^{j+1}))\right) \cdot L.
\end{align*}

Thus there is $j_q \ge 0$ with $2^{j_q} \in [\min(\wt{T}_q),\max(\wt{T}_q)]$ such that for $y_q := 2^{j_q}$, \[ R(\wt{T}_q \cap [y_q, 2y_q)) \gg \frac{\eta}{L (\log N)^{\delta} (\log \log N)^2}\frac{1}{\log y_q} \ge \frac{\Gamma}{\log y_q}. \]
Let $q_1, q_2 \in \mc{D}_h$. Define $y = \min\{y_{q_1}, y_{q_2}\}$ and let
\[ \mc{P} := \{\exp((\log N)^{\eps}) \le p \le \exp((\log y)(\log N)^{-\eps})\}. \]
Note that $\mc{P}$ depends on $q_1, q_2$. We should note that $\Gamma/\log y \ll 1$ from the above, so $\log y \ge \Gamma/C \gg (\log N)^{3\eps}$, so the set $\mc{P}$ is nonempty.
Let $\mc{P}_q = \{p \in \mc{P} : p \nmid n \forall n \in \wt{T}_q \cap [y_q, 2y_q]\}$ for $q \in \{q_1, q_2\}$. By \cref{lem:fund-lem} we know that
\[ R(\mc{P}_q) \le -\log\left(\frac{\Gamma}{C\log y_q}\right), \] for $q \in \{q_1, q_2\}$. Thus, we conclude that
\begin{align*}
    R(\mc{P} \setminus (\mc{P}_{q_1} \cup \mc{P}_{q_2})) &\ge R(\mc{P}) - R(\mc{P}_{q_1}) - R(\mc{P}_{q_2}) \\
    &\ge \log\left(\frac{\log y}{2(\log N)^{2\eps}} \right) + \log\left(\frac{\Gamma}{C\log y_{q_1}} \right) + \log\left(\frac{\Gamma}{C\log y_{q_2}} \right) \\
    &= \log\left(\frac{\Gamma^2}{2C^2(\log N)^{2\eps} \log \max(y_{q_1}, y_{q_2})} \right) \\
    &\ge \log\left(\frac{\Gamma^2}{2C^2(\log N)^{2\eps}(\log(N/M) + (\log N)^{1-\delta})} \right) \ge 1,
\end{align*}
where the final line uses that $y_q \le N/(qd_q)$ and $qd_q \ge M \exp(-(\log N)^{1-\delta})$ by \cref{lem:lift} for $q \in \{q_1, q_2\}$, and the assumption in the proposition statement.

Let $q \in \{q_1, q_2\}$. By definition, for each $p \in \mc{P} \setminus \mc{P}_q$ there exists $n/(qd_q) \in \wt{T}_q$ such that $p | (n/(qd_q))$. Hence $pqd_q | n$ and $|h_n| < K/2$. Because $|h_n| < K/2$, there is an $x \in I_h$ such that $n | x$, so $pqd_q | x$. Because $x_q$ is the unique multiple of $qd_q$ in $I_h = (h - K/2, h + K/2)$, we know that $x = x_q$, and thus $p | x_q$.
Therefore $p \in \mc{P} \setminus (\mc{P}_{q_1} \cup \mc{P}_{q_2})$ divides both $x_{q_1}$ and $x_{q_2}$, and hence $p | (x_{q_1} - x_{q_2})$. Because there are at least $\exp((\log N)^{\eps})$ such $p$ by the above, we conclude that $|x_{q_1} - x_{q_2}|$ is a multiple of some integer larger than $\exp((\log N)^{\eps})^{\exp((\log N)^{\eps})} > N$. Because $x_{q_1}, x_{q_2} \in I_h$, we know that $|x_{q_1} - x_{q_2}| \le K \le N$, so $x_{q_1} = x_{q_2}$ for all $q_1, q_2 \in \mc{D}_h$. Thus, there is a single $x \in I_h$ with $[\mc{D}_h] | x$.

\textbf{Minor arcs, finishing the proof:} 
We are now in position to complete the proof. Consider all $h\in (-Q/2,Q/2]$ such that $\mc{D}_h = \mc{D}$. As $[\mc{D}_h] | x$ for an element $x\in I_h$, we have that the number of such $h$ is bounded by 
\begin{equation}\label{eq:fin-bound}
(K + 1) \cdot \frac{[\mc{Q}_A]}{[\mc{D}]}\le N \cdot \prod_{q \in \mc{Q}_A \setminus \mc{D}} q \le N^{|\mc{Q}_A \setminus \mc{D}|+1}. 
\end{equation}
As we have restricted attention to $|h|>M/2\ge K/2$, we have that $\mc{D}_h \neq \mc{Q}_A$. Therefore the total contribution over minor arcs, using \cref{eq:minor-bound,eq:fin-bound} and is bounded by 
\[\frac{1}{Q}\sum_{\mc{D} \subsetneq \mc{Q}_A}N^{|\mc{Q}_A \setminus \mc{D}|+1} \cdot N^{-10|\mc{Q}_A \setminus \mc{D}|} \le \frac{1}{Q}\cdot \sum_{s\ge 1} N^{s+1} \cdot N^{s} \cdot N^{-10s} \le 2/(QN).\]
This completes the proof. 
\end{proof}

\section{Proof of \cref{thm:main}}
We are now in position to prove our main result using \cref{prop:main}. The first step is to find a range $[M, N]$ with good density, and $M = N/\exp((\log N)^{1-\alpha})$, on which we will ultimately apply \cref{prop:main}.
\begin{lemma}
\label{lem:m}
For $A \in [1, N]$ with $R(A) \ge \eta$  and $\alpha \in (0,3/4)$, there are $M', N'$ such that $M' \ge N'\exp((\log N')^{1-\alpha})$, and
\[ R(A \cap [M', N']) \gg \frac{\eta}{(\log N)^{\alpha} \log \log N}. \]
\end{lemma}
\begin{proof}
Define a sequence $N_1 = N$ and $\log N_{i+1} = \max(\log N_i - (\log N_i)^{1-\alpha}, 1)$. Note that $N_i = 1$ for $i \ge C(\log N)^{\alpha}\log\log N$, for sufficiently large constant $C$. Thus, there is an $i$ with
\[ R(A \cap [N_{i+1}, N_i]) \ge \frac{\eta}{C(\log N)^{\alpha}\log\log N}. \]
Thus, we can take $M' = N_{i+1}$ and $N' = N_i$.
\end{proof}
The next lemma prunes $A$ to ensure that the multiples of $q$ in $A$ have good density. This idea is essentially \cite[Lemma 6]{Blo21}.
\begin{lemma}
\label{lem:prune}
Assume $N$ is sufficiently large, and $\xi \in (0, 1)$. For $A \in [1, N]$ with $R(A) \ge \eta$, there is a subset $A' \subseteq A$ such that $R(A') \ge (1-\xi)\eta$, and for all $q \in \mc{Q}_{A'}$, that $q \cdot R(A'_q) \ge \frac{\eta\xi}{2 \log \log n}$.
\end{lemma}
\begin{proof}
Perform the following process: while there is $q \in \mc{Q}_A$ with $q \cdot R(A_q) < \frac{\eta\xi}{2 \log \log n}$, remove all multiples of $q$ from $A$. In total, this decreases $R(A)$ by at most
\[ \sum_{q \le N} \frac{1}{q} \frac{\eta\xi}{2 \log \log N} < \xi\eta. \qedhere\]
\end{proof}
We now prove our main proposition. 
\begin{proof}[Proof of \cref{thm:main}]
Initially, $A \subseteq [1, N]$ with $R(A) \ge (\log N)^{4/5+2\eps_0}$. Let $M', N'$ be as given by \cref{lem:m} for $\alpha = 1/5$. Relabel $A$ to $A \cap [M', N']$, $M'$ to $M$, and $N'$ to $N$. Note that now, $R(A) \ge (\log N)^{3/5+\eps_0}$, and $M = N\exp(-(\log N)^{4/5})$.

Set $\delta = 1/5$, and $K = M \exp(-(\log N)^{1-\delta}) = N \exp(-2(\log N)^{4/5})$. Finally, choose $S = N\exp(-6(\log N)^{4/5})$. Let $\mc{S}$ be the set of integers with a divisor larger than $S$. Note that
\begin{align*}
R([M, N] \cap \mc{S}) &\le \sum_{\log M \le k \le \log N} e^{-k}|[e^k, e^{k+1}] \cap \mc{S}| \\
&\ll \sum_{\log M \le k \le \log N} \frac{\log(e^k/S)}{k} \ll \frac{(\log N)^{4/5} \log(N/M)}{\log N} \\
&\ll (\log N)^{3/5},
\end{align*}
where the second line uses \cref{lem:smooth}.
Thus, $R(A \setminus \mc{S}) \ge (\log N)^{3/5+\eps_0}/2$. Replace $A$ with $A \setminus \mc{S}$ henceforth.

Let $T = \{n \ge 1 : \Omega(n) > 5 \log \log N\}$. By \cref{lem:divisor}, we know that 
\[ R(T \cap [M, N]) \ll \sum_{\log M \le k \le \log N} \frac{1}{(\log M)^{3}} \ll (\log N)^{-2}. \]
Thus, $R(A \setminus T) \ge (\log N)^{3/5+\eps_0}/4$. We replace $A$ with $A \setminus T$ henceforth. Through these reductions, we have ensured that all $n \in A$ have $\Omega(n) \le 5 \log \log N$, and $n$ is $S$--smooth.

Finally, apply \cref{lem:prune} to $A$ with $\xi = 1/2$, and relabel the resulting set $A'$ as $A$ again. We have that $R(A) \ge (\log N)^{3/5+\eps_0}/8$, and $q \cdot R(A_q) \ge \frac{(\log N)^{3/5+\eps_0}}{4 \log \log N} \ge (\log N)^{3/5+3\eps_0/4}$, for all $q \in \mc{Q}_A$. Let $\eta = (\log N)^{3/5+\eps_0/2}$.

We are in a position to apply \cref{prop:main}. Set all parameters as are done above, and $\eps = \eps_0/5$.
Recall that $\delta = 1/5$. We now verify the conditions. The lower bound on $\Gamma$ and upper bounds on $S, K$ are evident. For the final condition, we calculate that:
\begin{align*}
    &\frac{\Gamma^2}{(\log N)^{2\eps}(\log(N/M) + (\log N)^{1-\delta})} \ge \frac{\eta^2}{(\log N)^{2\eps}(\log N)^{2\delta}((\log N)^{4/5} + (\log N)^{1-\delta})(\log \log N)^6}.
\end{align*}
Because $\eta = (\log N)^{3/5+\eps_0/2}$, $\eps = \eps_0/5$, and $\delta = 1/5$, this expression is at least 
\[ (\log N)^{\eps_0/2}(\log \log N)^{-6} \ge (\log N)^{\eps_0/4}. \]
This verifies the fourth condition.

The conditions on elements $n \in A$ all hold by construction. Finally, we can choose $x = Q$, and note that $\log N \ge R(A) \ge 2$, so $R(A) \in [2x/Q, (\log N)x/Q]$. This completes the proof.
\end{proof}

\section{Proof of \cref{thm:300,prop:310}}
We now give the proof of \cref{thm:300}.
\begin{proof}[Proof of \cref{thm:300}]
Let $|A| \ge (1-1/e+\xi)N$, where $\xi > 0$ and $N$ sufficiently large in terms of $\xi$. Let $A' = A \cap [\xi N/2, N]$. Note that
\[ R(A') \ge \sum_{j=(1/e-\xi/2)N}^N \frac{1}{j} \ge 1+\frac{\xi}{10}. \]
Relabel $A'$ as $A$.

Set $\delta = 1/5$, $M = \xi N/2$, $K = M \exp((\log N)^{1-\delta})$, and $S = N \exp(-6(\log N)^{1-\delta})$. Remove from $A$ all numbers that are not $S$-smooth. As in the proof of \cref{thm:main}, by \cref{lem:smooth} this decreases $R(A)$ by $\ll (\log N)^{-1/5}\log(N/M) \le \xi/100$. Similarly, removing all $n$ with $\Omega(n) > 5 \log \log N$ reduces $R(A)$ by $\ll (\log N)^{-2} \le \xi/100$, by \cref{lem:divisor}. After this, we are left with a set $A'$ with $R(A') \ge 1+\xi/20$. Relabel $A'$ as $A$ again.

Now, apply \cref{lem:prune} to get a set $A'$ with $R(A') \ge 1+\xi/40$, and $q R(A'_q) \ge \frac{\xi}{100 \log \log N} := \eta$. Relabel this $A'$ as $A$ one final time.

We are now in a position to apply \cref{prop:main}. Set $\eps = 1/100$, $x = Q$, and note that $\Gamma \gg_{\xi} (\log N)^{3/5}(\log \log N)^{-7}$.
To verify the final condition, calculate that
\[ \frac{\Gamma^2}{(\log N)^{2\eps}(\log(N/M) + (\log N)^{1-\delta})} \gg_{\xi} (\log N)^{1/5}, \] as desired. Thus there is a subset $B \subseteq A$ with $R(B) = 1$.
\end{proof}

We now conclude with the proof of \cref{prop:310}.
\begin{proof}[{Proof of \cref{prop:310}}]
Let $\alpha = (\log N)^{-1/7+\eps_0}$. Because $\alpha \ge (\log N)^{-1}$, there exists a dyadic interval $[N', 2N']$ with $N'\ge N (\log N)^{-2}$ such that 
\[|A\cap [N',2N']|\ge \alpha \cdot N'/2.\]
Replace $N$ with $2N'$ and $A$ with $A \cap [N', 2N']$. Note that $R(A) \ge \alpha/4$.

Set $\delta = 1/7$, $M = N/2$, $K = M \exp(-(\log N)^{1-\delta})$ and $S = N \exp(-6(\log N)^{1-\delta})$. Removing all non-$S$-smooth numbers from $A$ reduces $R(A)$ by $\ll (\log N)^{-\delta}\log(N/M) \le \alpha/20$, by \cref{lem:smooth}. Removing all $n \in A$ with $\Omega(n) > 5 \log \log N$ decreases $R(A)$ by $\ll (\log N)^{-2}$, by \cref{lem:divisor}. After this, our set $A$ still has $R(A) \ge \alpha/8$.
Finally, prune $A$ using \cref{lem:prune} for $\xi=1/2$ so that for all $q \in \mc{Q}_A$, that $q \cdot R(A_q) \ge \frac{\alpha}{32 \log \log N} := \eta$. After this, we have $R(A) \ge \alpha/16$.

Now we apply \cref{prop:main}, with $x$ chosen later, so that $\alpha/128 \le x/Q \le \alpha/32$.
Note that for $\eps = 1/1000$,
\[ \Gamma := \frac{\eta^2(\log N)^{1-\delta}}{\log(N/M)^2(\log \log N)^5} \gg \alpha^2(\log N)^{5/7}(\log \log N)^{-7}. \]
Thus, we can calculate that
\[ \frac{\Gamma^2}{(\log N)^{2\eps}(\log(N/M) + (\log N)^{1-\delta})} \gg \alpha^4(\log N)^{4/7}(\log \log N)^{-14} \ge (\log N)^{\eps_0}, \] for $\alpha = (\log N)^{-1/7+\eps_0}$.

Thus, there is a subset $B \subseteq A$ with $R(B) = x/Q$. By \cref{lem:fund-lem}, there is a prime $p$ in $10/\alpha^2\le p \le \exp(C/\alpha)$, such that there is $n \in A$ with $p | n$. Here we are using $|A| \ge \alpha N/16$, $A \subseteq [N/2, N]$, and thus there exists $p \in \mc{Q}_A$ of the specified size. We now finally select $x$ so that $x/Q = \frac{\lfloor R(A) \cdot p/5\rfloor }{p}$, as desired.
\end{proof}

We end by briefly explaining why $t\ge \exp(1/(C\alpha))$ is in fact necessary. Let $L = \exp(1/(C\alpha))$. By \cref{thm:Merten}, we have that 
\[\sum_{p\le L}\frac{1}{p} = \log\log L + O(1).\]
By applying \cref{lem:fund-lem}, we have that set of integers between $[N/2,N]$ with no prime factor smaller than $L$ has density $\gg 1/\log L$. Choosing the $C$ sufficiently large, this set has at least $\alpha N$ elements and yet is easily seen to have reciprocal sum at most $<1$. Therefore any fraction $s/t$ constructed as a nontrivial reciprocal sum from this set must have $t\ge L$; this completes the proof.

\bibliographystyle{amsplain0.bst}
\bibliography{main.bib}

\providecommand{\bysame}{\leavevmode\hbox to3em{\hrulefill}\thinspace}
\providecommand{\MR}{\relax\ifhmode\unskip\space\fi MR }
\providecommand{\MRhref}[2]{%
  \href{http://www.ams.org/mathscinet-getitem?mr=#1}{#2}
}
\providecommand{\href}[2]{#2}
\begin{thebibliography}{10}

\bibitem{BE76}
M.~N. Bleicher and P.~Erd{\H{o}}s, \emph{Denominators of {E}gyptian fractions},
  J. Number Theory \textbf{8} (1976), 157--168.

\bibitem{BE76b}
Michael~N. Bleicher and Paul Erd{\H{o}}s, \emph{Denominators of {E}gyptian
  fractions. {II}}, Illinois J. Math. \textbf{20} (1976), 598--613.

\bibitem{BloWeb}
Thomas~F. Bloom, \url{https://www.erdosproblems.com/all}.

\bibitem{Blo21}
Thomas~F Bloom, \emph{On a density conjecture about unit fractions},
  arXiv:2112.03726.

\bibitem{Cro99}
Ernest~S. Croot, III, \emph{On some questions of {{E}rd\H{o}s} and {G}raham
  about {E}gyptian fractions}, Mathematika \textbf{46} (1999), 359--372.

\bibitem{Cro01}
Ernest~S. Croot, III, \emph{On unit fractions with denominators in short
  intervals}, Acta Arith. \textbf{99} (2001), 99--114.

\bibitem{Cro03}
Ernest~S. Croot, III, \emph{On a coloring conjecture about unit fractions},
  Ann. of Math. (2) \textbf{157} (2003), 545--556.

\bibitem{EG79}
P.~Erd\H{o}s and R.~L. Graham, \emph{Old and new problems and results in
  combinatorial number theory}, Monographies de L'Enseignement Math\'{e}matique
  [Monographs of L'Enseignement Math\'{e}matique], vol.~28, Universit\'{e} de
  Gen\`eve, L'Enseignement Math\'{e}matique, Geneva, 1980.

\bibitem{Er50}
P\'{a}l Erd\"{o}s, \emph{On a {D}iophantine equation}, Mat. Lapok \textbf{1}
  (1950), 192--210.

\bibitem{JLR00}
Svante Janson, Tomasz {\L}uczak, and Andrzej Rucinski, \emph{Random graphs},
  Wiley-Interscience Series in Discrete Mathematics and Optimization,
  Wiley-Interscience, New York, 2000.

\bibitem{Kuo19}
Dimitris Koukoulopoulos, \emph{The distribution of prime numbers}, Graduate
  Studies in Mathematics, vol. 203, American Mathematical Society, Providence,
  RI, [2019] \copyright 2019.

\bibitem{Ste24}
Stefan Steinerberger, \emph{On a problem involving unit fractions},
  arXiv:2403.17041.

\bibitem{Vos85}
Michael~D. Vose, \emph{Egyptian fractions}, Bull. London Math. Soc. \textbf{17}
  (1985), 21--24.

\bibitem{Yok86}
Hisashi Yokota, \emph{On a conjecture of {M}. {N}. {B}leicher and {P}.
  {E}rd{\H{o}}s}, J. Number Theory \textbf{24} (1986), 89--94.

\bibitem{Yok88}
Hisashi Yokota, \emph{Denominators of {E}gyptian fractions}, J. Number Theory
  \textbf{28} (1988), 258--271.

\bibitem{Yok88b}
Hisashi Yokota, \emph{On a problem of {B}leicher and {E}rd{\H{o}}s}, J. Number
  Theory \textbf{30} (1988), 198--207.

\end{thebibliography}

\end{document}